\documentclass{article}

\usepackage{amssymb}
\usepackage{amsmath}

\usepackage{todonotes}

\newcommand{\Z}{\mathbb{Z}}
\newcommand{\G}{\mathcal{G}}

\newtheorem{theorem}{Theorem}
\newtheorem{conjecture}{Conjecture}

\title{Pythagorean walks on $\Z^2$}

\author{Jan Willemson (ORCID 0000-0002-6290-2099)\\
Cybernetica\\
Narva mnt 20, Tartu, 51009, Estonia\\
\texttt{jan.willemson@cyber.ee}}

\begin{document}

\maketitle

\begin{abstract}
	We consider an infinite graph with the vertex set $\Z^2$ and edges connecting the vertices iff the Euclidean distance between the respective points is an integer, and the points do not lie on the same horizontal or vertical. Equivalently, there must exist a Pythagorean triangle with the hypotenuse corresponding to the graph edge and the legs parallel to the axes.
	
	We prove that the diameter of this graph is $3$, but surprisingly it appears that the nodes at the maximal (graph) distance of $3$ apart seem to be only those that are geometrically very close to each other. It also appears that the paths of length $2$ connecting geometrically close nodes may need to go through geometrically very distant points.
	
	We prove a general relation that generates infinite series of length-$2$ paths, and present the results of our computer experiments. We conclude the paper with a general conjecture about the length-$2$ and length-$3$ paths. We have posed this conjecture to several of the current leading AI models. Remarkably, none of them managed to make any significant progress in proving it.
\end{abstract}

\textbf{Keywords}: Graphs, Pythagorean walks 

\newpage

Consider the plane integer lattice $\Z^2$. Allowing only integer-length steps between the lattice points, it is trivially possible to reach every point from any other in at most two steps: we have to make at most one horizontal and at most one vertical move.

However, the situation becomes much more interesting when we prohibit horizontal and vertical moves. Formally, let us consider the undirected graph $\G$ with the vertex and edge sets being respectively
\begin{align*}
    V(\G)&=\{(x,y):x,y\in\Z\}\,,\\
    E(\G)&=\{(A(x_1,y_1),B(x_2,y_2)): |AB|\in\Z, x_1\ne x_2, y_1\ne y_2\}\,,
\end{align*}
and ask what is the shortest path between every pair of vertices.

Since any translation by an integer coordinate vector is clearly an automorphism of $\G$, we can fix an arbitrary vertex as the initial one. Let us select the origin $O(0,0)$ as our universal starting point.

Let us denote the length of the shortest path from $A$ to $B$ as $d(A,B)$, and let $d(\G)$ be the diameter of $\G$, i.e. the length of the longest path possible between the nodes of $\G$ (assuming it exists).

In one step we can of course reach from $O$ exactly such vertices $(x,y)$ that $\sqrt{x^2+y^2}\in\Z$ and $xy\ne0$. From this description it is clear that if there exists an edge from $O$ to $(x,y)$, there also exists symmetrical edges from $O$ to the vertices $(\pm x,\pm y)$ and $(\pm y, \pm x)$.

Getting from $(0,0)$ to $(1,1)$ requires two steps. As $\sqrt{1^2+1^2}$ is not an integer, one step is not enough. However, two steps suffice:
\[
    (0,0) \xrightarrow{(4,-3)} (4,-3) \xrightarrow{(-3,4)} (1,1)\,.
\]

To get from $O(0,0)$ to $A(1,0)$, one needs at least three steps. As the vertices $O$ and $A$ are not connected by an edge, more than one step is required. If there would be a lattice point $P$ such that $|OP|$ and $|AP|$ are integers, the triangle inequality implies that 
\[
    ||OP|-|AP||\leq |OA|=1\,.
\] 

The equality $||OP|-|AP||=1$ would imply that $P$ lies on the same line (i.e. the $x$-axis) with $O$ and $A$, but the graph does not have the corresponding edges $(O,P)$ and $(A,P)$.  

On the other hand, if $||OP|-|AP||=0$, the triangle $OAP$ would be isosceles, so $P$ would need to lie on the perpendicular bisector of $OA$, i.e. on the line $x=\frac{1}{2}$. Thus it can not have integer coordinates, and we conclude that there is no two-step-path.

The three-step-path from $(0,0)$ to $(1,0)$ can be organized in multiple ways, for example
\[
    (0,0) \xrightarrow{(9,12)} (9,12) \xrightarrow{(-12,-9)} (-3,3) \xrightarrow{(4,-3)} (1,0)\,.
\]

It turns out that three steps are always enough.

\begin{theorem}\label{thm:dG3}
    $d(\G)=3$.
\end{theorem}
\textit{Proof.} We have to prove that every node can be reached from the origin in no more than three steps. 

We start from a simple observation that there is an edge from $O$ to $(x,y)$ in $\G$, there is also an edge from $O$ to $(kx,ky)$ for every $k\in\Z\setminus\{0\}$. Indeed, if $x^2+y^2$ is a perfect square, then so is the number $(kx)^2+(ky)^2=k^2(x^2+y^2)$, and if $x,y\ne0$, the same also holds for $kx$ and $ky$. Thus, every such node $(kx,ky)$ is reachable from $O$ in one step (and $O$ itself is reachable in zero steps).

More generally, if there are two nodes $A(x_1,y_1)$ and $B(x_2,y_2)$ such that $(O,A),(O,B)\in E(G)$, all the linear combinations of the form
\[
    k(x_1,y_1)+\ell(x_2,y_2)\quad(k,\ell\in\Z)
\]
are reachable from $O$ in at most two steps. Similarly, for every three nodes $A(x_1,y_1)$, $B(x_2,y_2)$ and $C(x_3,y_3)$ such that 
$(O,A),(O,B),(O,C)\in E(G)$, all the linear combinations of the form
\[
    k(x_1,y_1)+\ell(x_2,y_2)+m(x_3,y_3)\quad(k,\ell,m\in\Z)
\]
are reachable from $O$ in at most three steps.

Thus, it is sufficient to find three edges $(O,A),(O,B),(O,C)\in E(\G)$ such that the $\Z$-linear combinations of the corresponding vectors span the whole $\Z^2$. For that, in turn, it is sufficient to find three nodes so that the respective $\Z$-linear span contains the vectors $(1,0)$ and $(0,1)$. 

We will show that we can choose the nodes $A(3,4)$, $B(4,3)$ and $C(4,-3)$. First it is clear that the corresponding edges $(O,A),(O,B),(O,C)$ belong to $E(\G)$. The construction of $(1,0)$ was actually already given above:
\[
    3\cdot(3,4)-3\cdot(4,3)+1\cdot(4,-3)=(1,0)\,.
\]
On the other hand, one can readily verify that
\[
    4\cdot(3,4)-4\cdot(4,3)+1\cdot(4,-3)=(0,1)\,.
\]
Putting it all together, for any $(x,y)\in V(\G)$, we can write
\begin{align*}
	(x,y)&=x\cdot(0,1)+y\cdot(0,1)=\\
		 &=(3x+4y)\cdot(3,4)-(3x+4y)\cdot(4,3)+(x+y)\cdot(4,-3)\,.
\end{align*}
\hfill$\square$

It turns out that for majority of the nodes $A(x,y)\in V(\G)$, the distance between $O$ and $A$ is actually no more than $2$. The only exceptions seem to be (geometrically very close!) nodes $(1,0)$, $(2,0)$ and $(2,1)$ (together with their symmetrical counterparts).

We have already considered $(1,0)$. Let prove the distance claim for the remaining two cases as well.

\begin{theorem}
    $d(O(0,0),A(2,0))=3$, $d(O(0,0),B(2,1))=3$.
\end{theorem}
    \textit{Proof.}
    In both of the cases, the corresponding pairs $(O,A)$ and $(O,B)$ do not belong to $E(\G)$, so it remains to prove that there are no paths of length $2$. 
    
    Assume first that there exists a lattice point $P$ such that $OP$ and $AP$ are integers. Since $|OA|=2$, the triangle inequality this time gives us 
	\[
	||OP|-|AP||\leq |OA|=2\,.
	\] 
	Similarly to the proof of Theorem~\ref{thm:dG3}, the case $||OP|-|AP||=2$ would imply that the point $P$ lies on the $x$-axis, so there would be no edges $(O,P)$ and $(A,P)$.
	
	The case $||OP|-|AP||=0$ would again mean that the triangle $OAP$ is isosceles with $|OP|=|AP|$. Thus, the vertex $P$ should lie on the line $x=1$, having the coordinates $(1,y)$. It must have $y\ne0$ (as otherwise there would again be no edges $(O,P)$ and $(A,P)$). But now we get
	\[
		|OP|=\sqrt{x^2+y^2}=\sqrt{1^2+y^2}
	\]
	that can not be an integer for any integer $y\ne0$. (WLOG we can assume $y\geq1$, and then $y^2+1$ lies strictly between $y^2$ and $(y+1)^2$.)
	
	Thus $||OP|-|AP||=1$ is the only case remaining to analyse. Denote $P(x,y)$ and consider the expressions for the segment lengths:
	\[
		|OP|=\sqrt{x^2+y^2}\,,\quad |AP|=\sqrt{(x-2)^2+y^2}\,.
	\]
	As $x^2$ and $(x-2)^2$ have the same parity, so do $x^2+y^2$ and $(x-2)^2+y^2$, and consequently also $|OP|$ and $|AP|$ (as these must be integers!). Hence they can not differ by $1$, showing that the desired point $P$ does not exist.
	
    Consider now a lattice point $P$ and assume that $OP$ and $BP$ are integers. The triangle inequality this gives us 
	\[
	||OP|-|BP||\leq |OB|=\sqrt{5}\approx2.236\,.
	\] 
    As $||OP|-|BP||$ has to be an integer, its possible values are $0$, $1$ and $2$.

    If $||OP|-|BP||=0$, the triangle  $OBP$ is isosceles with $|OP|=|BP|$. Thus, the vertex $P$ should lie on the perpendicular bisector of the segment $OB$. This bisector has the equation $y=-2x+2.5$, and hence contains no points with integer coordinates. 
	
	In case $||OP|-|BP||=1$, denote $P(x,y)$ and consider the expressions for the respective segment lengths:
	
	\begin{align*}
		\left|\sqrt{x^2+y^2}-\sqrt{(x-2)^2+(y-1)^2}\right|&=1\,,\\
		\left(\sqrt{x^2+y^2}-\sqrt{(x-2)^2+(y-1)^2}\right)^2&=1\,,\\
		x^2+y^2-2\sqrt{x^2+y^2}\cdot\sqrt{(x-2)^2+(y-1)^2}+(x-2)^2+(y-1)^2&=1\,.
	\end{align*}
	Opening the parentheses and organising the terms gives
	\[
		2x^2+2y^2-4x-2y+4=2\sqrt{x^2+y^2}\cdot\sqrt{(x-2)^2+(y-1)^2}\,.
	\]
	Squaring both sides again and simplifying leaves us with
    \begin{equation}\label{eq:OPBP1}
        3x^2+4xy-8x-4y+4=0\,.
    \end{equation}
    Consider it as a quadratic equation in $x$:
    \[
        3x^2+4(y-2)x+(-4y+4)=0\,.
    \]
    When its solution $x$ is an integer, the discriminant 
    \[
        (4(y-2))^2-4\cdot3\cdot(-4y+4)=16y^2-16y+16=16(y^2-y+1)\,.
    \]
    must be a perfect square, i.e. $y^2-y+1$ must be a perfect square. This is only possible if $y=0$ or $y=1$ (in which cases $y^2-y+1=1$). If $y\geq2$, we have strict inequalities 
    \[
        (y-1)^2<y^2-y+1<y^2\,,
    \]
    so $y^2-y+1$ lies between two consecutive full squares. If $y<0$, we can consider $y'=-y>0$, obtaining a similar result for the expression $y'^2+y'+1$, which lies strictly between $y'^2$ and $(y'+1)^2$.

    If $y=0$, the equation~(\ref{eq:OPBP1}) becomes
    \[
        3x^2-8x+4=0
    \]
    with the only integer solution being $x=2$. However, there is no edge between the points $O(0,0)$ and  $P(2,0)$ in $\G$, so we get no solutions here.
    
    If $y=1$, the equation~(\ref{eq:OPBP1}) becomes
    \[
        3x^2-4x=0
    \]
    with the only integer solution being $x=0$. Again, there is no edge between the points $O(0,0)$ and  $P(0,1)$ in $\G$, so we get no solutions from here either.

    It remains to consider the case $||OP|-|BP||=2$, which gives the equation 
    \[
        \left|\sqrt{x^2+y^2}-\sqrt{(x-2)^2+(y-1)^2}\right|=2
    \]
    for $P(x,y)$. Now the parity of $x^2$ and $(x-2)^2$ is the same, but the parities of $y^2$ and $(y-1)^2$ differ. Thus the parities of $\sqrt{x^2+y^2}$ and $\sqrt{(x-2)^2+(y-1)^2}$ differ as well. It follows that their difference can not be $2$, completing the proof.

~\hfill$\square$

Now the really interesting (and currently open) question remains -- are there any other, substantially different pairs of vertices having the distance $3$ between them in graph $\G$?

We conducted extensive computer experiments, and the results suggest that the answer to this question might surprisingly be no. We do not have a complete construction at this point; however, we can prove distance $(\leq)2$ for certain infinite families of pairs of vertices of $\G$. 

\begin{theorem}\label{thm:gh}
	If $(a,b,c)$ is a Pythagorean triple, then for any node $(g,h)\in V(\G)$ ($g,h\ne0$) such that
	\[
		(c\pm a)g=(c\pm b)h -1
	\]
	there exists a path of length $2$ from $(0,0)$ to $(g,h)$ in $\G$.
\end{theorem}
\textit{Proof.} If $(a,b,c)$ is a Pyhtagorean triple, then so is $(agh,bgh,cgh)$. We will prove that, under the relationship assumed in the theorem, 
\[
	(agh\pm g, bgh\pm h, cgh+(h-g))
\]
is also a Pythagorean triple. Combining the edges corresponding to these triples we are able to get from $(0,0)$ to some $(\pm g, \pm h)$ in two steps, and we already know that this is equivalent to being able to get to $(g,h)$.

Let us consider the equation needed for the stated triple to be Pythagorean, and transform it into an equivalent form.
\begin{align*}
	(agh\pm g)^2+(bgh\pm h)&\stackrel{?}{=}(cgh+(h-g))^2\,,\\
	(agh)^2\pm2ag^2h+g^2+(bgh)^2\pm2bgh^2+h^2&\stackrel{?}{=}(cgh)^2+2cgh(h-g)+h^2-2gh+g^2\,,\\ 
	\pm2ag^2h\pm2bgh^2&\stackrel{?}{=}2cgh(h-g)-2gh\,,\\
	\pm ag\pm bh &\stackrel{?}{=} c(h-g)-1\,,\\
	(c\pm a)g&\stackrel{?}{=}(c\pm b)h -1\,,
\end{align*}
and the latter is exactly the relation given in the statement of the theorem.

~\hfill$\square$

Theorem~\ref{thm:gh} allows us to generate many series of length-$2$ paths. For example, when $(a,b,c)=(4,3,5)$, the theorem implies that all the nodes $(g,h)$ with $(5-4)g=(5-3)h-1$ or $g=2h-1$ are achievable from $(0,0)$ in (at most) two steps. This includes the walks to $(1,1),(3,2),(5,3)$, etc.  

Similarly, when $(a,b,c)=(3,4,5)$, we get the equation $2g=h-1$, giving rise to length-$2$ walks from $(0,0)$ to $(1,3),(2,5),(3,7)$, etc. Of course, we can also get the equation $(5+3)g=(5-4)h-1$ or $8g=h-1$; the triple $(8,15,17)$, on the other, hand gives  $(17-8)g=(17-15)h-1$ or $9g=2h-1$, etc.

Two relatively simple special cases are presented in the next theorems.

\begin{theorem}\label{thm:n0}
	There exists a length-$2$ path from $(0,0)$ to $(n,0)$ for every $n\geq3$.
\end{theorem}
\textit{Proof.} 

First observe that for $n=4$ we have a path 
\[
(0,0) \xrightarrow{(9,12)}(9,12) \xrightarrow{(-5,-12)}(4,0)\,.
\]

Now let $n\geq3$ be odd. One can readily verify that
\[
    \left(n^2+\frac{n-1}{2},n^3-n,n^3-\frac{n-1}{2}\right)
    \quad\text{and}\quad
    \left(n^2-\frac{n+1}{2},n^3-n,n^3-\frac{n+1}{2}\right)
\]
are Pythagorean triples, with the first components differing by $n$ and second components being equal, hence allowing us to construct a path from $(0,0)$ to $(n,0)$. 

Now consider any integer $n\geq3$. We can present it in the form $n=2^km$, where $m$ is either $4$ or odd, and $k\geq0$. Thus, we can first have a construction for the path from $(0,0)$ to $(m,0)$, and then scale it by factor $2^k$. This concludes the proof.

~\hfill$\square$

\begin{theorem}\label{thm:n2n}
	There exists a length-$2$ path from $(0,0)$ to $(n,2n)$ for every $n\geq2$.
\end{theorem}
\textit{Proof.} 
First observe that for $n=2$ we have a path 
\[
(0,0) \xrightarrow{(77,-36)}(77,-36) \xrightarrow{(-75,40)}(2,4)\,.
\]
Now let $n\geq3$ be odd. One can readily verify that
\[
\left(n^2-n+2-\left(\frac{n-3}{2}\right)^2,n^2-n, n^2+\left(\frac{n-1}{2}\right)^2\right)
\]
and
\[
\left(n^2-2n+2-\left(\frac{n-3}{2}\right)^2,n^2+n, n^2+\left(\frac{n+1}{2}\right)^2\right)
\]
are Pythagorean triples, with the first components differing by $n$ and second components differing by $2n$, hence allowing us to construct a path from $(0,0)$ to $(n,2n)$. 

Now consider any integer $n\geq2$. We can present it in the form $n=2^km$, where $m$ is either $2$ or odd, and $k\geq0$. Thus, we can first have a construction for a path from $(0,0)$ to $(m,2m)$, and then scale it by factor $2^k$. This concludes the proof.

~\hfill$\square$

We have conducted computer search considering all the Pythagorean triples of the form $(d(m^2-n^2),2dmn,d(m^2+n^2))$ for $1\leq m,n,d<600$ and $\gcd(m,n)=1$ (the general formula for all Pythagorean triples, see~\cite{burton07}, Theorem 12.1). We studied all the nodes $(g,h)\in V(\G)$, where $0\leq g,h<4000$. The only trips for which we did not manage to establish graph distance $\leq2$ from $(0,0)$ are the examples we already know, i.e. $(0,1),(0,2),(1,2)$, together with their symmetric counterparts.

What makes the situation fascinating is that there does not seem to be a clear pattern how to arrange the walks of length $2$. 
In order to get to (geometrically) close points, one sometimes has to take relatively long steps back and forth. For example the walk from $(0,0)$ to $(2,4)$ we saw in the proof of Theorem~\ref{thm:n2n} is actually the shortest possible with both steps having the length of $85$.

The shortest two-step path we found from $(0,0)$ to $(2111,569)$ (a trip to geometric distance of about $2186.3$) is
\[
(0,0) \xrightarrow{(-50643549,196449668)}(-50643549,196449668) \xrightarrow{(50645660,-196449099)}(2111,569)\,,
\]
with the steps having the lengths of $202,\!872,\!475$ and $202,\!872,\!451$, respectively.

The existence of such extreme examples suggests there is no simple general formula to deduce $x$ and $y$ from $g$ and $h$ for the path
\[
	(0,0)\to(x,y)\to(g,h)\,.
\]

Under our current knowledge, it is of course  also possible that there are some faraway nodes $(g,h)$ such that no $\leq2$-step path from $(0,0)$ exists. However, as we have found such paths for all the non-trivial cases $0\leq g,h<4000$, we consider such a possibility unlikely. 

Thus, let us conclude our study with the following conjecture.

\begin{conjecture}
	In graph $\G$, the only vertices at distance $3$ from $O(0,0)$ are $(1,0)$, $(2,0)$, $(2,1)$ and their symmetric counterparts.
\end{conjecture}

We have also posed this conjecture to several AI tools (ChatGPT 5.4 Thinking, Claude Sonnet 4.6 and Opus 4.7, Gemini 3 Thinking). None of them managed to make any significant progress in proving the key claim, i.e. that for all but the explicitly listed small cases, the distance between the vertices is $\leq2$.

Assuming a constructive proof exists, it should probably be split into a number of subcases similar to Theorem~\ref{thm:gh}. However, the existence of irregular examples like $(0,0)\to (2111,569)$ means that such a splitting may become very complex. Alternatively, a non-constructive proof is, in principle, conceivable, but the author has not yet managed to find one.

Thus the author encourages all the math enthusiasts to tackle this conjecture and improve upon his results!

\subsection*{Declaration of interest}

The author declares no conflicting interests.

\subsection*{Declaration of AI usage}

AI tool (ChatGPT 5.4 Thinking) has been used to perform the background study to find out if similar research is available in the literature. It was also used as a general computer algebra system to verify the algebraic transformations given in the proofs and as a programming helper to optimize the Pythagorean triples analysis scripts. Several AI tools (ChatGPT 5.4 Thinking, Claude Sonnet 4.6 and Opus 4.7, Gemini 3 Thinking) were used to attempt proving the central hypothesis of the paper, but all of them failed. 

\subsection*{Acknowledgments}

This paper has been supported by the Estonian Research Council under the Grant PRG2177.

\end{document}